\documentclass[a4paper,12pt]{article}

\usepackage{palatino}
\usepackage{mathpazo}
\usepackage{pifont}
\usepackage{mathrsfs}

\usepackage{amsmath, amssymb,amsthm, latexsym}
\usepackage{enumitem}

\usepackage[small]{titlesec} 

\usepackage{multirow}

\newtheorem{thm}{Theorem}[section]
\newtheorem{lemma}[thm]{Lemma}

\newtheorem{prop}[thm]{Proposition}

\newcommand{\proofref}[1]{\noindent {\emph{Proof of Theorem} \ref{#1}.\ }}

\theoremstyle{definition}
\newtheorem{example}[thm]{Example}
\newtheorem{definition}[thm]{Definition}
\newtheorem{que}[thm]{Question}

\newtheorem{remark}[thm]{Remark}

\newcommand{\rrel}{\mathcal{R}}
\newcommand{\jrel}{\mathcal{J}}
\newcommand{\lrel}{\mathcal{L}}
\newcommand{\hrel}{\mathcal{H}}
\newcommand{\drel}{\mathcal{D}}

\newcommand{\nat}{\mathbb{N}}

\newcommand{\set}[2]{\ensuremath{\{\: #1 \: :\: #2 \:\}}}

\title{Ideals and finiteness conditions for subsemigroups}

\author{R. Gray\thanks{Supported by an an EPSRC Postdoctoral Fellowship, and partially supported by FCT and FEDER, project POCTI-ISFL-1-143 of Centro de \'{A}lgebra da Universidade de Lisboa, and by the project PTDC/MAT/69514/2006.}, V. Maltcev, J.D. Mitchell, N. Ru\v{s}kuc}

\begin{document}

\maketitle

\begin{abstract} 
In this paper we consider a number of finiteness conditions for semigroups related to their ideal structure, 
and ask whether 
such conditions are preserved by sub- or supersemigroups with finite Rees or Green index.
Specific properties under consideration include stability, $\drel=\jrel$ and minimal conditions on ideals.
\smallskip

\end{abstract}

\section{Introduction}\label{sec:intro}

Significant information about a semigroup may be obtained by studying its ideal structure and various finiteness conditions related to it. Examples include the existence of minimal ideals, stability and the property of Green's relations $\jrel$ and $\drel$ coinciding. Such properties have been identified and investigated because of their usefulness in the study of finite semigroups; see \cite[Appendix~A.2]{QTheoryBook}. 
This has led to instances where theorems that were originally proved for finite semigroups have been extended to apply to wider classes. 

Our main interest here is in the study of infinite semigroups satisfying such finiteness properties relating to their ideal structure.  
In this context, it is natural to ask, given a semigroup satisfying a certain property, to what extent
it can be changed while still continuing to satisfy the property. 
For example, an obvious basic question is whether the property in question is preserved under operations such as adjoining an identity element, or a zero element. 
Taking this one step further, one can consider this behaviour under finite changes in the number of elements. 
This leads to the notion of \emph{Rees index}. 
The Rees index of a subsemigroup $T$ of as semigroup $S$ is defined simply as the cardinality of the complement $S \setminus T$.     
Rees index was originally introduced and investigated by Jura \cite{Jura1,Jura2,Jura3}. 
Since then, the theory has been developed and extended considerably, with results about Rees index appearing in \cite{Cain2008, Hoffmann1, Kambites2005, Malheiro2, ruskuc1, Ruskuc&Thomas, Wang1,wong11}.   

Although natural, this notion is very restrictive, and as such limits the applicability of results about Rees index. 
For instance, it is not hard to see that an infinite group cannot have any proper subgroups of finite Rees index. 
Recently, in \cite{Gray}, a new approach was proposed, 
encompassing both Rees index and group-theoretic index, which is at the same time natural and
strong enough to enable one to prove results about preservation of finiteness conditions. 
A subsemigroup $T$ of a semigroup $S$ is said to have finite \emph{Green index} if it acts on its complement $S \setminus T$ in $S$ with finite quotient, in both of its natural actions via left and right multiplication (see below for a more detailed definition). 
The definition of Green index may also be given in terms of relative Green's relations, in the sense of \cite{Wallace1963}; see also \cite{Day} for a discussion of relative Green's relations in the context of the theory of topological semigroups.

Since Green index arises from the theory of relative ideals,  it is natural to consider the behaviour of finiteness properties relating to ideals  under taking finite Green index subsemigroups or extensions. This is our aim here. 
Specifically, after introducing Green index in Section \ref{sec2}, we consider the following finiteness 
conditions: stability (Section \ref{sect_stable}), $\jrel=\drel$ (Section \ref{secJD}),
having finitely many ideals (Section \ref{sect_finiteideals}), minimal conditions (Section \ref{secmin}),
all ideals having finite Rees index (Section \ref{secfinRees}), global torsion (Section \ref{secglobal}),
and eventual regularity (Section \ref{secpireg}).
In the process we resolve several open problems originally posed in \cite{ruskuc1} (specifically Open Problems 11.4, 11.3(i) and 11.3(ii)).
Our main results are summarised in Tables~\ref{table1} and \ref{table2}. Of course, each of these results will fail to hold if the finite index assumptions are lifted.

\begin{figure} \footnotesize
\begin{center}
\begin{tabular}{|l|l|c|c|} \hline
\multicolumn{2}{|c|}{       }& 
\parbox{28mm}{\ \\ \centering finite Green index subsemigroups } 
& 
\parbox{28mm}{\ \\ \centering \quad finite Green index extensions } 
\\
\hline\hline
$\jrel = \drel$ &  $T$ arbitrary & 
\parbox{25mm}{\ \\ \centering \ding{55} \\   (Example~\ref{ex_bigone}) \vspace{1mm}} & 
\parbox{25mm}{     \centering \checkmark \\ (Theorem \ref{J=D})} 
\\
\cline{2-4}
& $T$ regular& 
\parbox{25mm}{\ \\  \centering \checkmark\\ (Theorem \ref{th:first}) \vspace{1mm}} &
\parbox{25mm}{\centering \checkmark \\ (Theorem \ref{J=D}) }
\\
\hline
\multicolumn{2}{|l|}{\parbox{26mm}{\ \\ stability }}
& 
\parbox{25mm}{\ \\  \centering \checkmark\\  (Theorem \ref{stable}) \vspace{1mm}} & 
\parbox{25mm}{\centering \checkmark \\ (Theorem \ref{stable})}  
\\
\hline
\multicolumn{2}{|l|}{\parbox{26mm}{\ \\ finitely many ideals}}
& 
\parbox{25mm}{\ \\  \centering \checkmark\\  (Theorem \ref{finiteideals}) \vspace{1mm} } & 
\parbox{25mm}{\centering \checkmark\\  (Theorem \ref{finiteideals})}  
\\
\hline
\multicolumn{2}{|l|}{\parbox{26mm}{\ \\ $\mathrm{min}_R$}}
& 
\parbox{25mm}{\ \\ \centering \checkmark\\  (Theorem \ref{th:min_R}) \vspace{1mm} } & 
\parbox{25mm}{\centering \checkmark\\  (Theorem \ref{th:min_R})}  
\\
\hline
\multicolumn{2}{|l|}{\parbox{26mm}{\ \\ $\mathrm{min}_J$}}
& 
\parbox{25mm}{\ \\ \centering \checkmark\\  (Theorem \ref{th:min_J}) \vspace{1mm} } & 
\parbox{25mm}{\centering \checkmark\\  (Theorem \ref{th:min_J})}  
\\
\hline
\multicolumn{2}{|l|}{\parbox{26mm}{\ \\ global torsion }}
& 
\parbox{25mm}{\ \\ \centering \checkmark\\  (Theorem \ref{globalidem}) \vspace{1mm} } & 
\parbox{25mm}{\centering \checkmark\\  (Theorem \ref{globalidem})}  
\\
\hline
\multicolumn{2}{|l|}{\parbox{26mm}{\ \\ eventual\\ regularity }}
& 
\parbox{25mm}{\ \\ \centering \checkmark\\  (Theorem \ref{pireg}) \vspace{1mm} } & 
\parbox{25mm}{\centering \checkmark\\  (Theorem \ref{pireg})}  
\\
\hline
\end{tabular} 
\end{center}
\caption{Summary of Green index results.}
\label{table1}
\end{figure}

\begin{figure} \footnotesize
\begin{center}
\begin{tabular}{|l|l|c|c|} \hline
\multicolumn{2}{|c|}{       }& 
\parbox{28mm}{\ \\ \centering finite Rees index subsemigroups } 
& 
\parbox{28mm}{\ \\ \centering \quad finite Rees index extensions } 
\\
\hline\hline
$\jrel = \drel$ &  $S$ arbitrary & 
\parbox{25mm}{\ \\ \centering \ding{55} \\   (Example~\ref{ex_bigone}) \vspace{1mm}} & 
\parbox{25mm}{     \centering \checkmark \\ (Theorem \ref{J=D})} 
\\
\cline{2-4}
& $S$ regular& 
\parbox{25mm}{\ \\  \centering \checkmark\\ (Theorem \ref{th:second}) \vspace{1mm}} &
\parbox{25mm}{\centering \checkmark \\ (Theorem \ref{J=D})} 
\\
\hline
\multicolumn{2}{|l|}{\parbox{26mm}{\ \\ all ideals have finite Rees index}}
& 
\parbox{25mm}{\ \\  \centering \checkmark\\  (Theorem \ref{rees}) \vspace{1mm}} & 
\parbox{25mm}{\ \\ \centering \ding{55} \\ (Remark \ref{remark1})}  
\\
\hline
\end{tabular} 
\end{center}
\caption{Summary of Rees index results.} 
\label{table2}
\end{figure}

\section{Green's relations, relative relations and index}
\label{sec2}

Classical Green's relations are a cornerstone of semigroup theory; their definition can be found in every semigroup monograph, such as \cite{Howie1} or \cite{QTheoryBook}.
They may be viewed as capturing the orbit structure with respect to the actions of a semigroup $S$ on itself by left- and right multiplication.
Relative Green's relations, introduced by Wallace \cite{Wallace1963}, arise by considering the analogous orbit structure  with respect to the action of a subsemigroup rather than the entire semigroup.

More specifically, let $S$ be a semigroup, and let $T$ be a subsemigroup of $S$.
Denote by $S^1$ the semigroup obtained from $S$ by adjoining an identity element.
The five \emph{relative Green's relations} on $S$ with respect to $T$ are defined as follows:
\begin{gather*}
u \rrel^T v  \Leftrightarrow uT^1 = vT^1, 
\quad 
u \lrel^T v \Leftrightarrow T^1u = T^1v,
\quad
u \jrel^T v \Leftrightarrow  T^1uT^1 = T^1vT^1, \\
\hrel^T = \rrel^T \cap \lrel^T, 
\quad
\drel^T = \rrel^T\circ \lrel^T = \lrel^T \circ \rrel^T. 
\end{gather*}
Each of these relations is an equivalence relation on $S$; 
 the (relative) equivalence classes of an element $u\in S$ will be denoted by
$R_u^T$, $L_u^T$, $J_u^T$, $H_u^T$ and $D_u^T$ respectively.
Furthermore, each of these relations respects $T$, in the sense that every
relative class lies wholly
in $T$ or wholly in $S \setminus T$. 

The following result summarises some basic facts about relative Green's relations
(see \cite[Proposition~4]{Gray} for details).

\begin{prop}\label{basicproperties}
Let $S$ be a semigroup and let $T$ be a subsemigroup of $S$.

\begin{enumerate}[label=\textup{(\roman*)}, leftmargin=*,widest=iii]
\item
 $\rrel^T$ is a left congruence on $S$, and $\lrel^T$ is
a right congruence.
\item
For each relative $\hrel^T$-class $H$ either $H^2 \cap H =
\emptyset$, or $H^2 \cap H = H$, in which case $H$ is a subgroup of
$S$.
\item
Let $u,v \in S$  be such that $u \rrel^T v$, and let $p, q \in T^1$
such that $up = v$ and $vq = u$. Then the mapping $\rho_p$ given by
$x \mapsto xp$ is an $\rrel^T$-class preserving bijection from
$L_u^T$ to $L_v^T$, while the mapping $\rho_q$ given by $x \mapsto
xq$ is an $\rrel^T$-class preserving bijection from $L_v^T$ to
$L_u^T$, and is the inverse of the mapping $\rho_p$.
\end{enumerate}
\end{prop}

Following \cite{Gray}, we define the \emph{Green index} of $T$
in $S$ to be one more than the number of $\hrel^T$-classes in $S \setminus T$.
Thus, $T$ has \emph{finite Green index} in $S$ if there are only finitely many $\hrel^T$-classes in $S\setminus T$, or, equivalently, if $S\setminus T$ contains only finitely many 
$\rrel^T$- and $\lrel^T$-classes.
From this it is obvious that a subsemigroup with finite Rees index must also have finite Green index.
If $S$ is a group, and $T$ a subgroup, the relative $\rrel^T$- and $\lrel^T$-classes are precisely the left- and right cosets of $T$. Thus, for subgroups of groups, finite Green index coincides with the usual meaning of finite index.

Classical \emph{Green's relations} on $S$ are obtained by setting $T=S$ in the above.
They and the corresponding equivalence classes are normally written without superscripts, e.g. $\rrel$ and $R_u$. 
However,
since in this paper important roles will be played by both Green's equivalences and their relative versions, a peculiar notational difficulty arises. Given a semigroup $S$, a subsemigroup $T$, 
and 
$\mathcal{G}\in\{ \rrel,\lrel,\hrel,\drel,\jrel\}$,
there are three versions of $\mathcal{G}$: the `full' relation on $S$, the `full' relation on $T$, and the relative
relation $\mathcal{G}^T$ on $S$. In order to resolve this formally we would need to introduce another super- or subscript, to denote the domain of the relation in question. We have adopted a slightly more informal approach:
whenever $\mathcal{G}$ appears in the text (and there is a possibility of confusion) we will always specify its domain in words (e.g. $\mathcal{G}$ on $T$, or $\mathcal{G}^T$ on $S$); the occurrences of $\mathcal{G}$ in mathematical expressions will always be accompanied by the appropriate superscript $S$ or $T$, indicating from which set the relevant multiplying elements are drawn, while the actual domain of the relation in such a situation is always possible to determine from the context.

Associated with Green's equivalences $\rrel$, $\lrel$ and $\jrel$ on $S$ are
three natural preorders $\leq_\rrel$, $\leq_\lrel$, and $\leq_\jrel$ on $S$ given by
\[
u \leq_\rrel v \ \text{ if } \ u S^1 \subseteq v S^1,
\quad 
u \leq_\lrel v \ \text{ if } \ S^1 u \subseteq S^1 v,\quad 
u \leq_\jrel v \ \text{ if } \ S^1 u S^1 \subseteq S^1 v S^1.
\]
These preorders induce, in the natural way, partial orders on the set $S / \rrel$, $S / \lrel$  and $S / \jrel$, of of $\rrel$-, $\lrel$- and $\jrel$-classes respectively. 
These will all be simply denoted by $\leq$, and which one is meant will be clear from the context.

\section{Stability}\label{sect_stable}

Stable semigroups (originally introduced in \cite{Koch1957}) are important because they are precisely those semigroups for which the Rees--Sushkevich Theorem gives a coordinatization for each $\jrel$-class. 
Stability is also a useful tool for proving that a semigroup satisfies the finiteness condition $\jrel=\drel$. 
In particular, finite, torsion, or compact Hausdorff topological semigroups are all stable. Important results regarding stability include \cite{Birget1988, OCarroll1969}, and more recently \cite{Elston2002}. For more background on stable semigroups see \cite[Appendix A.2]{QTheoryBook}.

We recall the following
definition from \cite[Proposition 3.7]{lallement}.

\begin{definition}\label{defnstable}
A $\jrel$-class $J$ of a semigroup $S$ is said to be \emph{right
stable} if it satisfies one (and hence all) of the following
equivalent conditions:
\begin{enumerate}
\item[\rm (i)] the set of all $\rrel$-classes in $J$ has a minimal element with respect to $\leq_\rrel$;
\item[\rm (ii)] there exists $q\in J$ satisfying the following property:
 $q\jrel qx$ if and only if $q\rrel qx$ for all $x\in S$;
\item[\rm (iii)] every $q\in J$ satisfies the property stated in (ii);
\item[\rm (iv)] every $\rrel$-class in $J$ is minimal under $\leq_\rrel$ in the set of $\rrel$-classes in $J$.
\end{enumerate}
We say that the whole semigroup $S$ is \emph{right stable} if every
$\jrel$-class of $S$ is right stable. The notion of left stability
is defined dually. A $\jrel$-class or a semigroup are said to be
(\emph{two-sided}) \emph{stable} if they are both left and right
stable.
\end{definition}

The main theorem of this section is:

\begin{thm}\label{stable}
 Let $S$ be a semigroup and let $T$ be a subsemigroup of $S$ with finite Green index. Then
$T$  is (right, left or two-sided) stable if and only if $S$ is (right, left, or two-sided respectively) stable.
\end{thm}

Clearly a semigroup $S$ is left (right) stable if and only if the semigroup $S^1$ is left (right) stable. Hence, without loss of generality, throughout this section we will assume that $S$ has an identity $1$ and that $1 \in T$.  

We will need two technical lemmas.

\begin{lemma}[{\cite[Proposition 3.10]{lallement}}]\label{easy}
Let $S$ be a semigroup. Then $S$ is right stable if and only if
$R_a\leq R_{ba}$ implies $R_a=R_{ba}$ for all $a,b\in S$.
\end{lemma}

\begin{lemma}\label{step1}
Let $S$ be a semigroup, let $T$ be a right stable subsemigroup of
$S$ with finite Green index, and let $a,x\in S$ such that $(ax^i,
ax^j)\not\in \rrel^{T}$ for all $i\not= j$. Then there exists
$N\in \nat$ such that $x^i\in T$ and $(ax^i, ax^{2i})\not\in
\jrel^T$ for all $i\geq N$.
\end{lemma}

\begin{proof}
Since $(ax^i, ax^j)\not\in \rrel^{T}$ and $\rrel^T$ is a left congruence it follows that $(x^i,x^j)\not\in\rrel^T$ for all $i\neq j$.
As there are only finitely many $\rrel^T$-classes in $S\setminus T$ it follows that
there exists $N\in \nat$ such that $x^i,ax^i\in T$ for all $i\geq N$.
Hence the right ideal $ax^{i}T$ of $T$ properly contains the right ideal $ax^{2i}T$ for all $i\geq N$.
 It follows that $R_{ax^{i}}^T>R_{ax^{2i}}^T$.
Since $T$ is right stable, and recalling Definition~\ref{defnstable}(iv), $ax^{i}$ and $ax^{2i}$ lie in distinct $\jrel$-classes of $T$. That is,
$(ax^i, ax^{2i})\not\in \jrel^T$
 for all $i\geq N$.
\end{proof}

\proofref{stable}
We prove the theorem for right stability; the proof
for left stability is dual, and for the two-sided follows from these two.

($\Rightarrow$)
Suppose that $T$ is right stable.
It suffices, by Lemma \ref{easy}, to prove that if $R_a^{S}\leq R_{ba}^S$, then $R_a^{S}= R_{ba}^S$ for all $a,b\in S$. 
So suppose we have $a,b,x\in S$ such that
\begin{equation*}
a=bax=b^iax^i \ \ (i \in \nat).
\end{equation*}
We start by proving that there exist $i,j\in\nat$ such that $i< j$
and $(ax^i, ax^j)\in \rrel^T$. 
Seeking a contradiction, assume to the
contrary that $(ax^i, ax^j)\not\in \rrel^T$ for all $i\not=j$.
It follows from Lemma \ref{step1} that there exists $N\in\nat$ such
that $(ax^i, ax^{2i})\not\in\jrel^T$ and $x^i\in T$ for all $i\geq N$. 
From this and
$$
ax^i=b^i\cdot a x^{2i}\cdot 1,\ \ ax^{2i}=1\cdot ax^i\cdot x^i,
$$
we deduce that $b^{i}\in S\setminus T$ for all $i\geq N$. 
Since $T$ has finite Green index,
there exist $m,n\in\nat$ such that $m-n, n\geq N$ and $(b^m, b^n)\in \lrel^T$, 
and so there exists $t\in T$
such that $b^m=tb^n$. Hence 
$$
a=b^max^m=tb^nax^m=t\cdot a
x^{m-n}\cdot 1,\ \ ax^{m-n}=1\cdot a\cdot x^{m-n},
$$ 
and so
$(a, ax^{m-n})\in \jrel^{T}$. Similarly,
$$
a=t^2\cdot ax^{2(m-n)}\cdot1,\ \ ax^{2(m-n)}=1\cdot a\cdot x^{2(m-n)}
$$
implies that $(a, ax^{2(m-n)})\in \jrel^{T}.$ Therefore
$(ax^{m-n}, ax^{2(m-n)})\in \jrel^T$, a contradiction as $m-n\geq
N$.

So, we have shown that there exist $i<j$ such that $(ax^i, ax^j)\in\rrel^T$. In particular, there exists $u\in T$ such that $ax^i=ax^j u$. It follows that
$$ba=b^{i+1}ax^i=b^{i+1}ax^ju=ax^{j-i-1}u.$$
Thus from the assumption that $R_{a}^S\leq  R_{ba}^S$  we obtain $(ba,a)\in \rrel^{S}$. That is, $R^{S}_{ba}=R^{S}_{a}$, as required.\vspace{\baselineskip}

($\Leftarrow$) 
Suppose now that $S$ is right stable.
We prove
that  $R_a^{T}\leq R_{ba}^T$ implies $R_a^{T}= R_{ba}^T$ for all
$a,b\in T$.  Let $a,b,x\in T$ be such that
$a=bax=b^kax^k$. Since $R_a^S\leq R_{ba}^S$ and $S$ is right stable,
it follows that $R_a^S=R_{ba}^S$. Hence there exists $y\in S$ such
that $ba=ay$ (and so $b^ka=ay^k$ for all $k\geq 1$).
Now,
\begin{equation}\label{1star}
ba=b^{k+1}ax^k=ay^{k+1}x^k.
\end{equation}

If $y^{k+1}x^k\in T$ for some $k\geq 1$, then $ba\in a T$ by \eqref{1star}.
Hence $R_a^T=R_{ba}^T$ and the proof is complete.

On the other hand suppose that $y^{k+1}x^k\in S\setminus T$ for all $k\geq 1$. Then $y^k\in S\setminus T$ for all $k\geq 2$ (as $x\in T$). Then,
since $T$ has finite Green index, there exist $m\geq 2$ and $n\geq 1$ such that
$(y^{m+n}, y^m)\in \lrel^T$. Hence
 there exists $t\in T$ such that $y^{m+n}=ty^m$. Then for all $k\geq 1$ we have that
\begin{equation}\label{2star}
t^ky^mx^{m+kn-1}=y^{m+kn}x^{m+kn-1}\in S\setminus T.
\end{equation}
It follows that $y^mx^{m+kn-1}\in S\setminus T$ for all $k\geq 1$ (as $t\in T$). Hence,
again since $T$ has finite Green index,
 there exist $u,v\in \nat$ such that $v>u+1$ and $(y^mx^{m+un-1},  y^mx^{m+vn-1})\in \rrel^T$, and so there exists $t_0\in T$ where
\begin{equation}\label{3star}
y^mx^{m+un-1}= y^mx^{m+vn-1}t_0.
\end{equation}
To conclude, we have
\begin{eqnarray*}
ba&=& ay^{m+un}x^{m+un-1}
=at^uy^mx^{m+un-1} 
=at^uy^mx^{m+vn-1}t_0\\
&=&at^uy^mx^{m+un-1}\cdot x^{(v-u)n}t_0
=ba\cdot x^{(v-u)n}t_0
=bax\cdot x^{(v-u)n-1}t_0\\
&=&ax^{(v-u)n-1}t_0\in aT,
\end{eqnarray*}
where \eqref{1star}, \eqref{2star}, \eqref{3star} have been used in the first three steps above.
Thus $R_{a}^T=R_{ba}^T$, as required.
\qed

\section{\boldmath The Property $\jrel=\drel$}
\label{secJD}

Many natural classes of semigroups have the property that the relations $\jrel$ and $\drel$ coincide. For instance, this is the case for the full transformation monoid of all maps from a set to itself, for the monoid of all linear transformations on a vector space, and also every stable (and in particular every finite) semigroup.

Given a semigroup $S$ and subsemigroup $T$ of finite Rees index, it was asked in \cite[Open Problem 11.4]{ruskuc1} whether it is true that the relations $\jrel$ and $\drel$ coincide in $S$ if and only if they coincide in $T$. In this section we will show that this problem has a positive solution in one direction, when passing from $T$ to $S$, even under the weaker assumption of finite Green index. On the other hand, rather surprisingly, we will see that the converse does not hold, by exhibiting a semigroup $S$ and subsemigroup $T$ such that $|S \setminus T|=1$, where the relations $\jrel$ and $\drel$ coincide in $S$ but do not coincide in $T$. However, we will see that by placing regularity assumptions on $S$ or $T$, respectively, positive results in this direction may be recovered.   

We being by establishing the following. 

\begin{thm}\label{J=D}
Let $S$ be a semigroup, and let $T$ be a subsemigroup of
$S$ with finite Green index. If $\jrel=\drel$ in $T$, then
$\jrel=\drel$ in $S$.
\end{thm}

In order to prove Theorem~\ref{J=D} we need some preparation. Let
$S$ be a semigroup and $T$ be a subsemigroup of $S$ with finite
Green index such that $\jrel=\drel$ in $T$. 
Note that
$\jrel=\drel$ in $S$ if and only if $\jrel=\drel$ in $S^1$. 
Hence, as in the previous section, throughout this section we
assume without of loss of generality that $S$ has an identity $1$
and that $1\in T$. 
For any pair $a,b\in S$ with $(a,b)\in\jrel^S$ define
$$
Q_{a,b}=\set{(x_1,x_2, y_1, y_2)\in S\times S\times S\times S}{a=x_1by_1 \textrm{ and } b=x_2ay_2}.
$$
Note that
\begin{equation}\label{4star}
\begin{split}
&(x_1,x_2, y_1, y_2)\in Q_{a,b}\Rightarrow\\
&(x_1(x_2x_1)^k, x_2, y_1(y_2y_1)^k, y_2),
(x_1, x_2(x_1x_2)^k, y_1, y_2(y_1y_2)^k)\in Q_{a,b}.
\end{split}
\end{equation}

\begin{lemma}\label{lemma1}
Let $a,b\in S$ such that $(a,b)\in\jrel^S$ and let $(x_1,x_2, y_1, y_2)\in Q_{a,b}$. Then:
\begin{enumerate}[label=\textup{(\roman*)}, leftmargin=*,widest=ii]
\item[\rm (i)]
if the set
$$
\{k\in \nat:x_1(x_2x_1)^k\in S\setminus T \textrm{ or }
x_2(x_1x_2)^k\in S\setminus T\}
$$ 
is infinite, then $(b, x_1b),(a, x_2a)\in \lrel^S$;
\item[\rm (ii)]
if the set 
$$
\{k\in \nat:y_1(y_2y_1)^k\in S\setminus T \textrm{ or }
y_2(y_1y_2)^k\in S\setminus T\}
$$ 
is infinite, then $(b, by_1),(a, ay_2)\in \rrel^S$.
\end{enumerate}
\end{lemma}

\begin{proof}
It suffices to prove (i) assuming $x_1(x_2x_1)^k\in S\setminus T$ for infinitely many $k$.
Because $T$ has finite Green index in $S$, we have
$$(x_1(x_2x_1)^k, x_1(x_2x_1)^{k+r})\in \lrel^T$$ 
for some $k,r>0$,
and so
there exists $t\in T$ such that $x_1(x_2x_1)^k=t x_1(x_2x_1)^{k+r}$.
Hence
$$b=(x_2x_1)^{k+1}b(y_1y_2)^{k+1}=x_2tx_1(x_2x_1)^{k+r}b(y_1y_2)^{k+1}=x_2tx_1(x_2x_1)^{r-1}b$$
and so $(b, x_1b)\in\lrel^S$.
Also
$$a=(x_1x_2)^{k+1}a(y_2y_1)^{k+1}=tx_1(x_2x_1)^{k+r}x_2a (y_2y_1)^{k+1}=t(x_1x_2)^r a$$
and so $(a, x_2a)\in \lrel^S$.
\end{proof}


\begin{lemma}\label{lemma2}
Let $a,b\in S$ be such that $(a,b)\in\jrel^S$. If there exists  $(x_1, x_2, y_1, y_2)\in Q_{a,b}$ with $x_1, x_2, y_1, y_2\in T$,  then $(a,b)\in \drel^S$.
\end{lemma}

\begin{proof}
From $(x_1,x_2,y_1,y_2)\in Q_{a,b}$ and $x_1,x_2,y_1,y_2\in T$ it follows that
$a\in T$ if and only if $b\in T$.
If $a,b\in T$
we have that $(a,b)\in \jrel^T$
and so $(a, b)\in\drel^T$ by assumption. Thus $(a, b)\in \drel^S$,
as required.

Consider now the case where $a,b\in S\setminus T$. It will suffice to prove that
$(b, x_1b)\in \lrel^{S}$ and $(x_1b, a)\in \rrel^S$. 
Since
$b=(x_2x_1)^kb (y_1y_2)^k$, $b \not\in T$ and $y_1, y_2 \in T$,
we have that $(x_2x_1)^kb\in S\setminus
T$ for all $k\geq 1$. 
Finite Green index implies that there exist $m,n\geq 1$ and $t\in T$
with $(x_2x_1)^mb=t(x_2x_1)^{m+n}b$. 
Then
$$b=(x_2x_1)^mb(y_1y_2)^m=t(x_2x_1)^{m+n}b(y_1y_2)^m=t(x_2x_1)^nb$$
and so $(b,x_1b)\in \lrel^S$.
Analogously, $(b, by_1)\in \rrel^{S}$ and since $\rrel^{S}$ is a
left congruence, 
$$(a, x_1b)=(x_1by_1, x_1b)\in \rrel^S,
$$ 
completing the proof.
\end{proof}

\begin{lemma}\label{lemma3}
Let $a,b\in S$ such that $(a,b)\in\jrel^S$. If $(x_1, 1, y_1,
y_2)\in Q_{a,b}$ with $y_1, y_2\in T$, then $(a,b)\in \drel^S$.
\end{lemma}

\begin{proof}
As $(x_1, 1, y_1, y_2)\in Q_{a,b}$, we have that $a=x_1by_1$ and
$b=ay_2$. 
By \eqref{4star} we have
that $(x_1^k, 1, y_1(y_2y_1)^{k-1}, y_2)\in Q_{a,b}$ for all $k\geq 1$. 
Hence
if  there exists $k\in \nat$ such that $x_1^k\in T$, then
$(a,b)\in\drel^S$ by Lemma \ref{lemma2}.

Thus we may assume that $x_1^k\in S\setminus T$ for all $k\geq 1$.
Finite Green index of $T$ in $S$ implies that there exist $m,n\geq 1$ such that
$x_1^{m+n}=tx_1^m$ for some $t\in T$. 
Hence, since $x_1^m b (y_1 y_2)^m = b$, we have
$$
a=x_1^{m+n}b(y_1y_2)^{m+n-1}y_1=tx_1^{m}b(y_1y_2)^{m+n-1}y_1=t\cdot b\cdot (y_1y_2)^{n-1}y_1.
$$
It follows that $(t, 1, (y_1y_2)^{n-1}y_1, y_2)\in Q_{a,b}$, and, since all the entries are in $T$, the result follows
by Lemma \ref{lemma2}.
\end{proof}

The following lemma provides the crucial step in the proof of
Theorem \ref{J=D}.

\begin{lemma}\label{lemma4}
Let $a,b\in S$ such that $(a,b)\in\jrel^S$. If $(x_1, x_2, y_1,
y_2)\in Q_{a,b}$ with $y_1, y_2\in T$, then $(a,b)\in \drel^{S}$.
\end{lemma}

\begin{proof}
There are two cases to consider: 
\begin{enumerate}[label=\textup{(\arabic*)}, leftmargin=*,widest=0]
\item there exists $N\in \nat$ such that $x_1(x_2x_1)^k,x_2(x_1x_2)^k\in T$ for all $k\geq N$; and
\item $x_1(x_2x_1)^k\in S\setminus T$ or $x_2(x_1x_2)^k\in S\setminus T$ for infinitely many $k$.
\end{enumerate}
In Case (1), the quadruple
$$(x_1(x_2x_1)^N, x_2(x_1x_2)^N, y_1(y_2y_1)^N, y_2(y_1y_2)^N)$$ lies in $Q_{a,b}$ and all of its entries are in $T$. Hence the result follows by Lemma \ref{lemma2}.

To prove the lemma in Case (2), note that $x_2a=x_2x_1\cdot b\cdot
y_1$ and $b=1\cdot x_2a\cdot y_2$. This implies that $(x_2a, b)\in
\jrel^S$ and  $(x_2x_1, 1, y_1, y_2)\in Q_{x_2a, b}$.  So, by Lemma
\ref{lemma3}, $(x_2a, b)\in \drel^S$. By the assumption of Case (2)
it follows from Lemma \ref{lemma1}(i) that $(x_2a, a)\in\lrel^S$.
Therefore $(a,b)\in\drel^S$.
\end{proof}


We can now use Lemmas \ref{lemma1} and \ref{lemma4} to prove Theorem
\ref{J=D}. \vspace{\baselineskip}


\begin{proof}[Proof of Theorem~\ref{J=D}]
Let $a,b\in S$ such that $(a,b)\in \jrel^{S}$.
Then by Lemma \ref{lemma4} (and its dual), if  there exists $(x_1, x_2, y_1,
y_2)\in Q_{a,b}$ with either $x_1, x_2\in T$ or $y_1, y_2\in T$,
then the proof is complete.

If neither of these conditions hold,  then for all $(x_1, x_2, y_1,
y_2)\in Q_{a,b}$ and  for all  $k\in \nat$ we have 
\[
\bigl(x_1(x_2x_1)^k\in S\setminus T \mbox{ or }
x_2(x_1x_2)^k\in S\setminus T\bigr) \mbox{ and }
\bigl( y_1(y_2y_1)^k\in S\setminus T \mbox{ or }
y_2(y_1y_2)^k\in S\setminus T\bigr)
\]
by \eqref{4star}.  Therefore by
Lemma \ref{lemma1}, $(b, x_1b)\in \lrel^S$ and $(b, by_1)\in
\rrel^S$. Thus $(a, by_1)=(x_1b\cdot y_1, b\cdot y_1)\in \lrel^S$ and so
$(a, b)\in \drel^S$.
\end{proof}

The property $\jrel=\drel$ is not inherited the other way round,
from $S$ to $T$, even when $T$ has finite Rees index, as the following
example shows. 

\begin{example}
\label{ex_bigone}

\renewcommand{\arraystretch}{1.4}

\begin{figure}[t]
{\footnotesize
\begin{eqnarray*}
&
\begin{array}{c|ccccc}
     & a & b & c & d & x \\ \hline\hline
x   &  xa  & xb   &  0  & 0   & x^2   \\ 
x^2   & a   &  x^2b  & 0   & 0   & x  \\ 
a   &  0  & 0   & ac   & ad   & 0  \\ 
ac^i   &  0  & 0   &  ac^{i+1}  & ac^{i-1}   & 0  \\ 
ad^j   &  0  & 0   &  ad^{i-1}  & ad^{i+1}   & 0  \\ 
c^i d^j & 0 & 0 & c^{i+1} d^j & c^i d^{j+1} & 0 \\
xa   &  0  & 0   & xac   & xad   & 0  \\ 
xac^i   &  0  & 0   &  xac^{i+1}  & xac^{i-1}   & 0  \\ 
xad^j   &  0  & 0   &  xad^{i-1}  & xad^{i+1}   & 0  \\ 
x^j (bx)^i b^k x^l   &  
\mbox{see below}
  &    
\mbox{see below}
 & 0   &  0  &  x^j (bx)^i b^k x^{l}x \\ 
\end{array}
&
\\
&
x^j (bx)^i b^k x^l\cdot a =\left\{ 
\begin{array}{ll}
ac^i, & \mbox{if } k=l,\ j=0,2\\
xac^i, & \mbox{if } k=l,\ j=1 \\
0, & \mbox{otherwise}
\end{array}\right. 
\quad
x^j (bx)^i b^k x^l\cdot b=\left\{ 
\begin{array}{ll}
x^j (bx)^i b^k x^lb, & \mbox{if } k=l\\
0, & \mbox{otherwise}
\end{array}\right. 
&
\end{eqnarray*}
\caption{The action of the generators on the elements of $S$ via right multiplication in Example~\ref{ex_bigone}.}
\label{fig_right}
}
\end{figure}

\renewcommand{\arraystretch}{1.4}

\begin{figure}[t]
{\footnotesize
\begin{eqnarray*}
&
\begin{array}{l|cccccccccc}
 & x & x^2 & a  &  ac^i & ad^j   &   c^i d^j & xa  & xac^i  &  xad^j  & x^j (bx)^i b^k x^l \\ \hline\hline
a   & 0 & 0 & 0 & 0 & 0 & ac^{i-1} d^{i-1} & 0 & 0 & 0 & 0 \\
b   & bx & bx^2 & 0 & 0 & 0 & 0 & ac & ac^{i+1} & ad^{j-1} & 
\mbox{see below}
\\
c   & 0 & 0 & 0 & 0 & 0 & c^{i+1} d^j & 0 & 0 & 0 & 0 \\
d    & 0 & 0 & 0 & 0 & 0 & c^{i} d^{j+1} & 0 & 0 & 0 & 0 \\
x   & x^2 & x & xa & xac^i & xad^j & 0 & a & ac^i & ad^j & xx^j (bx)^i b^k x^l
\end{array}
&
\\
&
b\cdot x^j (bx)^i b^k x^l =\left\{
\begin{array}{ll}
b x^j (bx)^i b^k x^l, & \mbox{if } j=1 \mbox{ or } j=i=k=0,\\
0, & \mbox{otherwise}
\end{array}
\right.
&
\end{eqnarray*}
\caption{The action of the generators on the elements of $S$ via left multiplication in Example~\ref{ex_bigone}.}
\label{fig_left}
}
\end{figure}

We are going to define a semigroup $S$ by means of a (fairly large) presentation.
The generators are
$$
A=\{a,b,c,d,x\},
$$
and the main relations are
\begin{equation}
\label{eq1b}
bxa=ac,\ acd=a,\ dc=cd,\ x^3=x,\ x^2a=a.
\end{equation}
There is also a number of zero relations, making the `unnecessary' products of generators equal to zero:
\begin{eqnarray*}
&&aa=ab=ax=0,\\
&&ba=bb=bc=bd=0,\\
&&ca=cb=cx=0,\\
&&da=db=dx=0,\\
&&xc=xd=0,\\
&&bx^2b=0.
\end{eqnarray*}
Note also that
\begin{equation*}
ac^{k+1}d=acdc^k=ac^k
\end{equation*}
for all $k\geq 0$. A routine check confirms that the presentation
together with the relations $ac^{k+1}d=ac^k$, viewed as a string rewriting system, is confluent. (See \cite{Book} for definitions relating to rewriting systems.) 
It is easy to see that
this rewriting system is also terminating: indeed, it is length reducing,
except for the relation $dc= cd$, which pushes $d$s systematically
to the right. Therefore, a set of normal forms is provided by all
the words from $A^+$ which do not contain the left hand side of a
relation as a subword; they are:
\begin{eqnarray*}
&& x,\ x^2\\
&&
ac^i,\ ad^j,\ xac^i,\ xad^j,\ c^id^j \ (i,j\geq 0),\\
&&
x^j (bx)^i b^k x^l\ (i\geq 0,\ j=0,1,2, \ k=0,1,\ l=0,1,2).
\end{eqnarray*}
(with the empty word excluded).

Computing the non-singleton Green's classes in $S$ we obtain:
\begin{eqnarray*}
&&
L_x^S=R_x^S=D_x^S=J_x^S=\{ x,x^2\},\\
&&
L_{ac^i}^S=\{ ac^i,xac^i\},\ L_{ad^j}^S=\{ ad^j,xad^j\},\\
&&
R_{a}^S=\{ac^i, ad^i\::\: i\geq 0\}, R_{xa}^S=\{xac^i,xad^i\::\: i\geq 0\},\\
&&
J_{a}^S=D_{a}^S=\{ ac^i,ad^i,xac^i,xad^i\::\: i\geq 0\}.
\end{eqnarray*}
The remaining non-singleton Green's classes in $S$ arise from the remaining normal form words that begin or end in $x$, that is, those of the form
$
x \alpha x, x \alpha b, $ or $b \alpha x, 
$ 
where $\alpha \in A^*$. These elements give rise to the following non-trivial Green's classes in $S$
\begin{eqnarray*}
&&
R_{x \alpha x}^S=\{x \alpha x, x \alpha x^2 \}, L_{x \alpha x}^S=\{x \alpha x, x^2 \alpha x \}, \\ 
&&
D_{x \alpha x}^S=J_{x \alpha x}^S = \{ x \alpha x, x \alpha x^2, x^2 \alpha x, x^2 \alpha x^2 \}, \\
%
%
&&
L_{x \alpha b}^S=D_{x \alpha b}^S=J_{x \alpha b}^S=\{x \alpha b, x^2 \alpha b \}, \\ 
&&
R_{b \alpha x}^S=D_{b \alpha x}^S=J_{b \alpha x}^S=\{b \alpha x, b \alpha x^2 \}. \\ 
%
%
%
%
\end{eqnarray*}
Two useful observations that can be used for the verification of these claims are:
\begin{itemize}
\item If $u$ and $w$ are non-zero words both representing non-zero elements of $S$, and $u=w$ in $S$, then $u$ and $w$ must contain the same number of occurrences of the letter $a$. 
\item If $u$ and $w$ are words both representing the same element of $S$, then $u$ contains a letter different from $x$ if and only if $w$ contains a letter different from $x$. 
\end{itemize}
The claims above about Green's classes $\rrel$, $\lrel$ and $\drel$ in $S$  can all now be easily read off from Tables~\ref{fig_right} and \ref{fig_left}. Of the remaining claims, the most important is that $J_a^S = D_a^S$ so let us now see why this is so. Clearly $D_a^S \subseteq J_a^S$. For the converse, suppose that $w \jrel^S a$ where $w$ is a normal form word. This means there are normal form words $\alpha, \beta, \gamma, \delta$ such that
\[
\alpha w \beta = a \  \mbox{and} \  \gamma a \delta = w
\]
in $S$. 
From $\alpha w \beta = a$ it follows that the word $\alpha w \beta$ contains exactly one occurrence of the letter $a$. But $\gamma a \delta = w$ tells us that $w$ contains at least one occurrence of the letter $a$. Therefore, $w$ must contain exactly one occurrence of the letter $a$, and thus looking at the list of normal form words we conclude that $w$ belongs to the set 
\[
\{ ac^i,ad^i,xac^i,xad^i\::\: i\geq 0\} = D_{a}^S. 
\]
Therefore  $J_a^S = D_a^S$. The claims about the remaining non-trivial $\drel$- and $\jrel$-classes are easily verified, and we conclude $\jrel^S=\drel^S$.

Let now $T=S\setminus\{x\}$. The only words of $S$ that are equal to $x$ are $x^{2i+1}$, where $i\geq 1$.
Such a word cannot be expressed as a product of two elements of $T$. Hence $T$ is a subsemigroup of $S$.
Now note that
$$
a=b\cdot xa \cdot d,\ xa= xbx \cdot a \cdot d;
$$
hence $(a,xa)\in \jrel^T$.
We claim that $(a,xa)\not\in \drel^T$.
As in $S$ we have
$$
R_a^T=\{ac^i,ad^i\::\: i\geq 0\}.
$$
However, unlike the situation in $S$, the $\lrel$-class of $xa$ in $T$ is trivial.
Indeed, 
looking at Table~\ref{fig_left} we see that 
the only elements of $T$ we can premultiply $xa$ with and not obtain $0$ are
of the form $b$, $(xb)^i$, $(xb)^i x^2$, $x(xb)^i$, $x(xb)^ix^2$.
After rewriting we obtain the words $ac^i$ and $xac^i$ where $i \geq 1$.
Thus, by premultiplying $xa$ by elements of $T$ we never get back to $xa$, and so $L_{xa}^T$ is trivial.
Therefore $L_{xa}^T\cap R_a^T=\emptyset$, and hence $(a,xa)\not\in\drel^T$.
\end{example}

The situation is made even more curious by the fact that the property $\jrel=\drel$ \emph{is} inherited by subsemigroups of finite Green index if certain regularity assumptions are made on $S$ \emph{or} $T$.
Below are two sample results. We have not been able to obtain a unified general result.

\begin{thm}\label{th:first}
Let $T$ be a regular subsemigroup of finite Green index in a semigroup $S$.
Then $\jrel=\drel$ in
$S$ implies $\jrel=\drel$ in $T$.
\end{thm}

\begin{thm}\label{th:second}
Let $T$ be a subsemigroup of finite Rees index in a regular semigroup $S$.
Then $\jrel=\drel$ in $S$
implies $\jrel=\drel$ in $T$.
\end{thm}

In order to prove  Theorems \ref{th:first} and  \ref{th:second} we need the following technical lemma:

\begin{lemma}\label{pr:aux}
Let $S$ be a semigroup, let $T$ be a subsemigroup of finite Green index in $S$, and let
$a,b\in T$ be such that $a\drel^S ab$. Then there exists $d\in
T$ such that $a\lrel^S d\rrel^S ab$ and $d\jrel^T a$.
\end{lemma}

\begin{proof}
Since $a\drel^S ab$, there exists $c\in S$ such that
$a\lrel^S c\rrel^S ab$. This means that there exist
$x_1,x_2,y_1,y_2\in S$ with
\[
a = x_1c, \ c = y_1a,\ ab = cx_2,\  c = aby_2.
\]
Then $a=x_1\cdot a\cdot by_2$ and $c=x_1\cdot c\cdot by_2$.

There are two cases to consider.
\smallskip

\textbf{Case 1:} \emph{there are infinitely many $k\geq 1$ with
$(by_2)^k\in S\setminus T$.}
Then, since $T$ has finite Green index in $S$, 
there exist $k,n\geq 1$ such that $(by_2)^k\rrel^T
(by_2)^{k+n}$. In particular there exists $t\in T$ such that
$(by_2)^{k+n}t=(by_2)^{k}$. Then
\begin{equation*}
a=x_1^ka(by_2)^k=x_1^ka(by_2)^{k+n}t=a(by_2)^nt=c\cdot
(by_2)^{n-1}t.
\end{equation*}
Together with $c=a\cdot by_2$, we obtain $a\rrel^S c$. Then
$a\rrel^S ab$ and the assertion holds with $d=a$.
\smallskip

\textbf{Case 2:} \emph{there exists $k_0\geq 1$ such that $(by_2)^k\in T$
for all $k\geq k_0$.}
We claim that the assertion of the lemma holds for $d=c$.
We prove first that $c\in T$. Suppose the converse: $c\in S\setminus
T$, and recall that $c=x_1^kc\cdot (by_2)^k$. Hence
$x_1^kc\in S\setminus T$ for all $k\geq k_0$. Then there exist
$k,n\geq k_0$ such that $x_1^kc\lrel^T x_1^{k+n}c$. In
particular, $tx_1^kc=x_1^{k+n}c$ for some $t\in T$. Then
\begin{equation*}
c=x_1^{k+n}c(by_2)^{k+n}=tx_1^{k}c(by_2)^{k+n}=tx_1\cdot
c\cdot (by_2)^{n+1}=t\cdot a\cdot (by_2)^{n+1}.
\end{equation*}
Since $n\geq k_0$, we obtain $c\in T$, a contradiction. Hence $c\in
T$. It remains to prove that $c\jrel^T a$.

Now, we have $c=x_1^k\cdot a\cdot (by_2)^{k+1}$ for all $k\geq k_0$.
If there are infinitely many $k$ such that $x_1^k\in S\setminus T$,
then there exist $k,n\geq k_0$ such that $x_1^{k+n}=t\cdot x_1^{k}$
for some $t\in T$. Then
$$c=x_1^{k+n}\cdot a\cdot (by_2)^{k+n+1}=tx_1^{k}\cdot a\cdot
(by_2)^{k+n+1}=t\cdot a\cdot (by_2)^{n+1}$$ and so $c\in TaT$. On
the other hand, if $x_1^k\in T$ for all $k\geq N_0$ for some
$N_0\geq k_0$, then $c=x_1^{N_0}\cdot a\cdot (by_2)^{N_0+1}$ and so
$c\in TaT$.

Having $a=x_1^{k+1}\cdot c\cdot (by_2)^k$ for all $k\geq k_0$, by
analogous reasoning as in the previous paragraph we deduce that
$a\in TcT$. Thus $c\jrel^T a$, as required.
\end{proof}

\proofref{th:first}
Suppose that $\jrel=\drel$ in $S$.
Take two $\jrel^T$-equivalent elements $t_1$ and $t_2$ of $T$. 
Then there exist elements $\alpha_1,\alpha_2,\beta_1,\beta_2\in
T$ with $t_1=\alpha_1t_2\beta_1$ and $t_2=\alpha_2t_1\beta_2$. Then
$t_1\jrel^T t_1\beta_2\jrel^T t_2$ and
$t_2=\alpha_2\cdot t_1\beta_2$. Hence, to prove the theorem, it
suffices to establish that, for every two elements $a,b\in T$, if
$a\jrel^T ab$ then $a\drel^T ab$, and that if
$a\jrel^T ba$ then $a\drel^T ba$. We will only prove the first assertion, the second follows by a similar argument. 

So suppose that $a,b\in T$ are such that $a\jrel^T ab$. Then
$a\jrel^S ab$ and so $a\drel^S ab$. By
Lemma~\ref{pr:aux} we have that $a\lrel^S c\rrel^S ab$ for some $c\in T$. Now, since $T$ is regular, we have
$a\lrel^T c$ and $c\rrel^T ab$ (see \cite[Proposition~A.1.16]{QTheoryBook}). Thus $a\drel^T ab$
and we are done.
\qed\vspace{\baselineskip}

Before proving Theorem \ref{th:second} we need another technical result:

\begin{lemma}\label{lm:funny}
Let $T$ be a semigroup and $x,y,\alpha,\beta,\gamma,\delta\in T$
such that $x=\alpha y\beta$ and $y=\gamma x\delta$. Furthermore,
assume that there exists $n\in\mathbb{N}$ such that
$y=y(\beta\delta)^n$. Then $x$ and $y$ are $\drel$-related in
$T$.
\end{lemma}

\begin{proof}
First notice that $y\rrel^T y\beta$. Now,
\begin{equation*}
y=(\gamma\alpha)^n\cdot y\cdot (\beta\delta)^n=(\gamma\alpha)^ny.
\end{equation*}
Hence $y\beta=(\gamma\alpha)^{n-1}\gamma\cdot \alpha
y\beta=(\gamma\alpha)^{n-1}\gamma\cdot x$. Since $x=\alpha\cdot
y\beta$ we obtain $y\beta\lrel^T x$. Thus $x\drel^T y$.
\end{proof}

\proofref{th:second}
Suppose that $\jrel=\drel$ in $S$.
As in the proof of Theorem~\ref{th:first}, it suffices to prove that
if $a\jrel^T ab$ then $a\drel^T ab$ for all $a,b\in T$.
So, let $a\jrel^T ab$ for some $a,b\in T$. Then
$a\drel^S ab$. By Lemma~\ref{pr:aux} we have that there
exists $c\in J_a^T=J_{ab}^T$ such that $a\lrel^S c\rrel^S ab$.

Therefore, in order to prove the theorem, it is enough to prove that
if $x,y\in T$ are such that $x\rrel^S y$ (or $x\lrel^S y$)
and $x\jrel^T y$, then $x\drel^T y$. We will do this
only in the case of $\rrel$, the other case follows by symmetry. 

So, let $x,y\in T$ be such that $x\jrel^T y$ and
$x\rrel^S y$.
Since $S$ is regular, 
there
exist $x',y'\in S$ such that $x=xx'x$ and $y=yy'y$. There also exist
$\alpha,\beta,\gamma,\delta\in T$ such that $x=\alpha y\beta$ and
$y=\gamma x\delta$. There are four cases.

\vspace{\baselineskip}

\noindent\textbf{Case 1:} \emph{$x\in yT$ and $y\in xT$.}
Then immediately $x\rrel^T y$, as required.

\vspace{\baselineskip}

\noindent\textbf{Case 2:}  \emph{$x=yt$ and $y=xf$ for some $t\in T$ and
$f\in S\setminus T$.} In this case we distinguish three subcases:

\vspace{\baselineskip}

\noindent\textbf{Subcase 2a:} \emph{$x'\in T$ and $y'\in T$.} Then $x\rrel^T xx'$ and
$y\rrel^T yy'$. In addition, $xx'\rrel^S x\rrel^S
y\rrel^S yy'$. Hence, since an idempotent is a left identity in its $\rrel$-class, $xx'=yy'\cdot xx'$ and $yy'=xx'\cdot
yy'$. Therefore $xx'\rrel^T yy'$ and so $x\rrel^T y$.

\vspace{\baselineskip}

\noindent\textbf{Subcase 2b:} \emph{$x'\in S\setminus T$ and $y'\in T$.} Then
$y\rrel^T yy'$. Moreover,
\begin{align}\label{eq:case}
x &= yy'\cdot yt,\\
\label{eq:case1}
yy' &= x\cdot fy'.
\end{align}
If $fy'\in T$ then $x\rrel^T yy'$ and so $x\rrel^T y$.
So, suppose $fy'\in S\setminus T$. Recall that $x\rrel^S yy'$ and
$x\jrel^T yy'$. Since it suffices to prove that
$x\drel^T yy'$ and since $yy'$ is an idempotent, in view of~\eqref{eq:case} we may assume that $y^2=y=y'$. Now \eqref{eq:case1} becomes $y=x\cdot
fy$. If $fy\in T$ then $x\rrel^T y$, as required. Hence we may
assume that $fy\in S\setminus T$. Then $f(\gamma\alpha)^k\cdot
y(\beta\delta)^k\in S\setminus T$ for all $k\geq 1$ and so, since $\beta, \delta \in T$,
it follows that $f(\gamma\alpha)^k\in S\setminus T$ for all $k\geq 1$. 
Since $T$ has finite Rees index in $S$,
this implies
that $f(\gamma\alpha)^k=f(\gamma\alpha)^{k+n}$ for some $k,n\geq 1$.
Then
\begin{equation*}
fy=f\cdot
(\gamma\alpha)^{k+n}y(\beta\delta)^{k+n}=f(\gamma\alpha)^{k}y(\beta\delta)^{k+n}=fy(\beta\delta)^{n}.
\end{equation*}
Hence $y=x\cdot fy=xfy(\beta\delta)^{n}=y(\beta\delta)^n$ and so by
Lemma~\ref{lm:funny}, $x\drel^T y$, as required.

\vspace{\baselineskip}

\noindent\textbf{Subcase 2c:}  \emph{$y'\in S\setminus T$.} Then
$y=x\cdot fy'y$. If $fy'y\in T$ then $x\rrel^T y$. Hence we
may assume that $fy'y\in S\setminus T$. Then
$fy'(\gamma\alpha)^k\cdot y(\beta\gamma)^k\in S\setminus T$ for all
$k\geq 1$. Thus $fy'(\gamma\alpha)^k\in S\setminus T$ for all $k\geq
1$. Since $S\setminus T$ is finite, there exist $k,n\geq 1$ such that
$fy'(\gamma\alpha)^k=fy'(\gamma\alpha)^{k+n}$. Hence
\begin{equation*}
fy'(\gamma\alpha)^k=fy'(\gamma\alpha)^{k+nr}
\end{equation*}
for all $r\geq 1$. Since $y=xf$, we have $yy'(\gamma\alpha)^k=yy'(\gamma\alpha)^{k+nr}$ for all $r\geq 1$. Now,
\begin{equation}
\label{11}
y = yy'y = yy'(\gamma\alpha)^ky(\beta\delta)^k= yy'(\gamma\alpha)^{k+nr}y(\beta\delta)^k 
= yy'(\gamma\alpha)^{nr}y 
\end{equation}
for all $r\geq 1$. Hence we may assume that
$y'(\gamma\alpha)^{nr}\in S\setminus T$ for all $r\geq 1$ (otherwise
the assertion follows by Subcase 2b and we obtain that $x\drel^T y$). So,
since $S \setminus T$ is finite,
there exist $r_1<r_2$ with $r_2-r_1>1$ such that
$y'(\gamma\alpha)^{nr_1}=y'(\gamma\alpha)^{nr_2}$. Then
\begin{equation}
\label{22}
\begin{split}
y'y(\beta\delta)^n &= y'(\gamma\alpha)^{nr_1}\cdot y(\beta\delta)^{n(r_1+1)} 
= y'(\gamma\alpha)^{nr_2}\cdot y(\beta\delta)^{n(r_1+1)}  \\
&= y'(\gamma\alpha)^{n(r_2-r_1-1)}y.
\end{split}
\end{equation}
Combining \eqref{11} and \eqref{22} yields
\begin{equation*}
y=y\cdot y'(\gamma\alpha)^{n(r_2-r_1-1)}y=yy'y\cdot
(\beta\delta)^n=y(\beta\delta)^n
\end{equation*}
and so $x\drel^T y$ by Lemma~\ref{lm:funny}.

\vspace{\baselineskip}

\noindent\textbf{Case 3:} \emph{$x=yf$ and $y=xt$ for some $t\in T$ and
$f\in S\setminus T$.}
This case is similar to Case 2.

\vspace{\baselineskip}

\noindent\textbf{Case 4:} \emph{$x=yf_1$ and $y=xf_2$ for some $f_1,f_2\in
S\setminus T$.} Once again we will distinguish three subcases:

\vspace{\baselineskip}

\noindent\textbf{Subcase 4a:} \emph{$x'\in T$ and $y'\in T$.} 
Observe that $xx'\rrel^S yy'$, so that $xx'\rrel^T yy'$, by properties of idempotents, and it follows immediately that $x\rrel^T y$.

\vspace{\baselineskip}

\noindent\textbf{Subcase 4b:} \emph{$x'\in S\setminus T$ and $y'\in T$.} Note first that
\[
x = yy'\cdot yf_1,\ 
yy' = x\cdot f_2y'.
\]
If either of $yf_1$ or $f_2y'$ is in $T$ then $x\drel^T yy'$ by Cases 1--3, and we are done.
Since $y\rrel^T yy'$,
then $x\drel^T y$ and we are done. Hence 
$yf_1 \in S \setminus T$ and $f_2 y' \in S \setminus T$, and
without loss of
generality we may assume that $y^2=y$ and $y'=y$. Then $y=x\cdot
f_2y$. If $f_2y\in T$ then we reduce to Case 3 and the proof is
complete. So we may assume that $f_2y\in S\setminus T$. Then, as
before, $f_2(\gamma\alpha)^k\in S\setminus T$ for all $k\geq 1$.
Then $f_2(\gamma\alpha)^{k+n}=f_2(\gamma\alpha)^k$ for some $k,n\geq
1$. This implies $f_2y(\beta\delta)^n=f_2y$ and so
\begin{equation*}
y=xf_2y=xf_2y(\beta\delta)^n=y(\beta\delta)^n.
\end{equation*}
Then by Lemma~\ref{lm:funny}, $x\drel^T y$.

\vspace{\baselineskip}

\noindent\textbf{Subcase 4c:} \emph{$y'\in S\setminus T$.} Then in the same way as in Case
2c one can show that this subcase can be reduced to Case
4b or Case 3.
\qed


\section{Finitely Many Ideals}\label{sect_finiteideals}

In~\cite{Gray} it was proved that if $T$ is a subsemigroup of
finite Green index in a semigroup $S$, then $T$ has finitely many
right (respectively, left) ideals if and only if $S$ has finitely many right (resp., left) ideals. In this section we prove the corresponding theorem for the case of two-sided ideals. In particular, this provides a positive solution to \cite[Open Problem 11.3(i)]{ruskuc1}.

\begin{thm}\label{finiteideals}
Let $S$ be a semigroup and let $T$ be a subsemigroup of $S$ with
finite Green index. Then $T$ has finitely many ideals if and only if
$S$ has finitely many ideals.
\end{thm}

\begin{proof}
As usual, we assume without loss of generality that $S$ has an identity element and that $1\in T$.

($\Rightarrow$)
Suppose that $T$ has finitely many ideals, or, equivalently, finitely many $\jrel$-classes.
Let $J^S$ be an arbitrary $\jrel$-class of $S$.
Then $J^S\cap T$ is a union of $\jrel$-classes of $T$, while $J^S\cap (S\setminus T)$ is a union of relative $\rrel^T$-classes of $S$. It follows that $S$ has finitely many $\jrel$-classes.

($\Leftarrow$)
Let now $S$ have finitely many ideals, and suppose that $T$ has infinitely many ideals.
Then there exists a $\jrel$-class $J_u^S$ of $S$ which contains infinitely many $\jrel$-classes of $T$. 
In particular, $J_u^S$ either contains an infinite chain
or an infinite antichain
of $\jrel$-classes of $T$. In either case, for an arbitrary $N\in\mathbb{N}$ we can pick $u_1,\ldots,u_N\in J_u^S\cap T$ such that
\begin{equation}
\label{eq1c}
J_{u_i}^T \not\leq  J_{u_j}^T\ (1\leq i<j\leq N).
\end{equation}

We choose a specific $N$ as follows.
Let $P+1$ be the Green index of $T$ in $S$; thus, $P$ is equal to the number of $\hrel^T$-classes in $S\setminus T$.
Let $Q=P^2+2$, and
let $N$ be the Ramsey number $R(Q,Q,Q)$. Recall that this means that for every edge colouring of the complete graph of size $N$ with three colours there exists a monochromatic complete subgraph with $Q$ vertices.

Since $u_1,\ldots, u_N$ are all $\jrel$-related in $S$, we can write
\begin{equation}
\label{eq2b}
u_i = \alpha_{i} u_{i+1} \beta_{i}\ (i=1,\ldots,N-1;\ \alpha_i,\beta_{i}\in S).
\end{equation}
Define
\begin{equation}
\label{eq3b} \alpha_{i,j}=\alpha_i \alpha_{i+1} \ldots
\alpha_{j-1},\ \beta_{i,j}=\beta_{j-1}\ldots \beta_{i+1}\beta_i\ (1\leq i<j\leq N).
\end{equation}
These elements satisfy
\begin{equation}
\label{eq4b} u_i=\alpha_{i,j} u_j \beta_{i,j}\ (1\leq i< j\leq N).
\end{equation}

From (\ref{eq4b}) and (\ref{eq1c}) it follows that for all 
$1\leq i<j\leq N$ at least one of $\alpha_{i,j},\beta_{i,j}$ is not in $T$. 
Recalling $N=R(Q,Q,Q)$, it follows by Ramsey's Theorem that there exists a set $I\subseteq \{1,\ldots , N\}$
of size $Q$ such that one of the following three possibilities holds:
\begin{eqnarray}
\label{eqnrcase1}
&&\alpha_{i,j},\beta_{i,j}\in S\setminus T\ (i,j\in I,\ i<j),\\
\label{eqnrcase2}
&&\alpha_{i,j}\in S\setminus T,\ \beta_{i,j}\in T\ (i,j\in I,\ i<j),\\
\label{eqnrcase3}
&&\alpha_{i,j}\in  T,\ \beta_{i,j}\in S\setminus T\ (i,j\in I,\ i<j).
\end{eqnarray}
Furthermore, by discarding the elements of $\{1,\ldots,N\}$ that do not belong to $I$, and re-indexing, we may take
$$
I=\{1,\ldots,Q\}.
$$

Suppose first that (\ref{eqnrcase1}) holds.
Each of the $Q-1=P^2+1$ pairs $(\alpha_{i,Q},\beta_{i,Q})$ ($1\leq i<Q$) belongs to $(S\setminus T)\times (S\setminus T)$.
Since the number of $\hrel^T$-classes in $S\setminus T$ is precisely $P$, it follows by the Pigeonhole Principle that for some $1\leq i<j<Q$ we have
\begin{equation}
\label{eq10}
(\alpha_{i,Q},\alpha_{j,Q}),(\beta_{i,Q},\beta_{j,Q})\in\hrel^T,
\end{equation}
and write
\begin{equation}
\label{eq11} \alpha_{i,Q}=a\alpha_{j,Q},\ \beta_{i,Q}=\beta_{j,Q}b\
(a,b\in T).
\end{equation}
Now we have
$$
\begin{array}{rcll}
u_i &=& \alpha_{i,Q} u_Q \beta_{i,Q} & \mbox{(by (\ref{eq4b}))}\\
&=& a\alpha_{j,Q} u_Q \beta_{j,Q}b & \mbox{(by (\ref{eq11}))}\\
&=& a u_j b & \mbox{(by (\ref{eq4b}))},
\end{array}
$$
contradicting (\ref{eq1c}).

Suppose now that (\ref{eqnrcase2}) holds.
Again using the Pigeonhole Principle, this time applied to $Q-1$ elements $\alpha_{i,Q}\in S\setminus T$ ($1\leq i<Q$),
we see that
there exist $i,j$ ($1\leq i<j<Q$) such that
\begin{equation}
\label{eq7} (\alpha_{i,Q},\alpha_{j,Q})\in\hrel^T.
\end{equation}
Let $a\in T$ be such that
\begin{equation}
\label{eq8} \alpha_{i,Q}=a\alpha_{j,Q}.
\end{equation}
Now we have
$$
\begin{array}{rcll}
u_i &=& \alpha_{i,Q}u_Q\beta_{i,Q} & \mbox{(by (\ref{eq4b}))} \\
&=& a\alpha_{j,Q}u_Q\beta_{i,Q} & \mbox{(by (\ref{eq8}))} \\
&=& a\alpha_{j,Q} u_Q \beta_{j,Q} \beta_{i,j}  &
   \mbox{(by (\ref{eq3b}))} \\
&=& a u_j  \beta_{i,j} & \mbox{(by (\ref{eq4b})).}
\end{array}
$$
But from (\ref{eqnrcase2}) we have $\beta_{i,j}\in T$, and this contradicts (\ref{eq1c}).
Case (\ref{eqnrcase3}) can be eliminated by a dual argument, and the theorem is proved.
\end{proof}


\section{Minimal Conditions for Ideals}
\label{secmin}

Recall that a semigroup $S$ is said to have property $\min_R$ (respectively $\min_J$) if every descending chain
$R_{x_1}^S\geq R_{x_2}^S\geq R_{x_3}^S\geq\cdots$ (respectively $J_{x_1}^S\geq J_{x_2}^S\geq
J_{x_3}^S\geq \cdots$) of $\rrel$- (respectively $\jrel$-) classes of $S$ eventually stabilizes. 
Obviously $\min_R$ and
$\min_J$ are finiteness conditions.

\begin{thm}\label{th:min_R}
Let $S$ be a semigroup and let $T$ be a subsemigroup of $S$ with
finite Green index. Then $T$ satisfies $\min_R$ if and only if $S$ satisfies
$\min_R$.
\end{thm}

\begin{proof}
Without loss of generality we may assume that $S$ has an
identity $1$ and $1\in T$.

($\Rightarrow$) 
Suppose $T$ satisfies $\min_R$, but that in $S$ we have an infinite decreasing chain
$R_{x_1}^S>R_{x_2}^S>R_{x_3}^S>\cdots$ of $\rrel$-classes.

If there are infinitely many elements from $S\setminus T$ among
$x_1,x_2,\ldots$, then there exist $i<j$ such that $x_i\rrel^T x_j$,
implying $x_i\rrel^S x_j$, a
contradiction. Hence there are only finitely many $i$ such that
$x_i\in S\setminus T$, and without loss of generality we may assume
that in fact $x_i\in T$ for all $i\geq 1$. 
Now, for every $n\geq 2$ there
exists $p_n\in S$ such that $x_{n-1}p_n=x_{n}$. Then $x_1\cdot
p_2\cdots p_i=x_{i}$ for all $i\geq 2$. If $p_2\cdots p_i\in
S\setminus T$ for all $i\geq 2$, then, 
since there are finitely many $\rrel^T$-classes in $S\setminus T$,
there would exist $1<i<j$ such
that $p_2\cdots p_i\rrel^T p_2\cdots p_j$ and so
$x_{1}=x_1p_2\cdots p_i\rrel^Sx_1p_2\cdots p_j=x_{j}$, a
contradiction. Hence there exists $i_2>i_1=1$ such that $p_2\cdots
p_{i_2}\in T$. Then $R_{x_{i_1}}^T\geq R_{x_{i_2}}^T$. Analogously,
there exists $i_3>i_2$ such that $R_{x_{i_2}}^T\geq R_{x_{i_3}}^T$.
Proceeding in this way, there exists an infinite sequence
$i_1<i_2<i_3<\cdots$ such that $R_{x_{i_1}}^T\geq R_{x_{i_2}}^T\geq
R_{x_{i_3}}^T\geq\cdots$. Since every $x_i$ lies in $T$ and $T$ satisfies
$\min_R$, we must have that $R_{x_{i_k}}^T=R_{x_{i_{k+1}}}^T$ for some
$k$. Then $R_{x_{i_k}}^S=R_{x_{i_{k+1}}}^S$, a contradiction.

($\Leftarrow$) 
Suppose $S$ satisfies $\min_R$, but that in $T$ we have an infinite descending chain
$R_{x_1}^T>R_{x_2}^T>R_{x_3}^T>\cdots$ where $x_i\in T$. 
Since
$R_{x_1}^S\geq R_{x_2}^S\geq R_{x_3}^S\geq\cdots$, we may assume
without loss of generality that $R_{x_n}^S=R_{x_{n+1}}^S$ for all
$n\geq 1$. Then for every $n\geq 1$ there exists $q_n\in S$ with
$x_{n+1}q_n=x_n$. Now,
\begin{equation*}
x_i=x_{i+1}q_i=\cdots=x_{n+1}q_n\cdots q_i
\end{equation*}
for all $1\leq i\leq n$. Hence $q_n\cdots q_i\in S\setminus T$ for
all $1\leq i\leq n$. Then there exist numbers $i<j<N$ such that
$q_N\cdots q_i\hrel^T q_N\cdots q_j$. In particular, there
exists $t\in T$ with $q_N\cdots q_i=q_N\cdots q_j\cdot t$. Then
\begin{equation*}
x_i=x_{N+1}q_N\cdots q_i=x_{N+1}q_N\cdots q_jt=x_jt,
\end{equation*}
a contradiction.
\end{proof}

\begin{remark}
The above proof does not use the full strength of the assumption that $T$ has finite Green index in $S$, i.e. that the number of $\hrel^T$-classes in $S\setminus T$ is finite, but only that there are finitely many $\rrel^T$-classes in $S\setminus T$.
\end{remark}

Now we will prove an analogue of Theorem~\ref{th:min_R} for
$\min_J$. For this we will require the following lemma.

\begin{lemma}\label{lm:chain}
Let $T$ be a subsemigroup of finite Green index in a semigroup $S$.
Let also $J_{x_1}^S>J_{x_2}^S>J_{x_3}^S>\cdots$ be an infinite descending chain
of $\jrel$-classes of $S$
where $x_i\in T$ for all $i\geq 1$. Then there is a sequence
$n_1<n_2<n_3<\cdots$ such that $J_{x_{n_1}}^T\geq J_{x_{n_2}}^T\geq
J_{x_{n_3}}^T\geq\cdots$.
\end{lemma}

\begin{proof}
For each $n\geq 1$ there exist $p_n,q_n\in S$ such that
$x_{n+1}=p_nx_nq_n$. Define $p_{i,j}=p_{j-1}\cdots p_i$ and
$q_{i,j}=q_i\cdots q_{j-1}$ for all $1\leq i<j$. Then
$x_j=p_{i,j}x_iq_{i,j}$ for all $1\leq i<j$. By Ramsey's Theorem
there exists an infinite subset $I\subseteq\mathbb{N}$ such that
$p_{i,j}\in T$ for all $i,j\in I$ with $i<j$, or $p_{i,j}\in
S\setminus T$ for all $i,j\in I$ with $i<j$; and $q_{i,j}\in T$ for
all $i,j\in I$ with $i<j$, or $q_{i,j}\in S\setminus T$ for all
$i,j\in I$ with $i<j$. By renumbering, without loss of generality we
may assume that $I=\mathbb{N}$. If all $p_{i,j}$ and $q_{i,j}$ are
from $T$, then $J_{x_1}^T\geq J_{x_2}^T\geq J_{x_3}^T\geq\cdots$ and
we are done. Hence suppose that all $p_{i,j}$ are from $S\setminus T$
(the case when
all $q_{i,j}$ are from $S\setminus T$ being analogous). 
Now consider two possible cases:

\vspace{\baselineskip}

\textbf{Case 1:} \emph{$q_{i,j}\in T$ for all $1\leq i<j$.} By
Ramsey's Theorem there exists an infinite subset
$J\subseteq\mathbb{N}$ such that all the $p_{i,j}$ with $i,j\in J$ and $i<j$ 
lie in the same $\hrel^T$-class. After renumbering
we may assume that $J=\mathbb{N}$. Then, in particular,
$p_{n+1}p_n=p_{n,n+2}\hrel^T p_{n,n+1}=p_n$ for all $n\geq 1$.
Hence there exists $t_{n+1}\in T$ such that $p_{n+1}p_n=t_{n+1}p_n$.
Then
\begin{equation*}
x_{n+2}=p_{n+1}p_nx_nq_nq_{n+1}=t_{n+1}p_nx_nq_nq_{n+1}=t_{n+1}x_{n+1}q_{n+1}\in
Tx_{n+1}T
\end{equation*}
for all $n\geq 1$. Therefore $J_{x_2}^T\geq J_{x_3}^T\geq
J_{x_4}^T\geq\cdots$.

\vspace{\baselineskip}

\textbf{Case 2:} \emph{$q_{i,j}\in S\setminus T$ for all $1\leq
i<j$.} By the Pigeonhole Principle there exist numbers $N<i<j$ such
that $p_{N,i}\hrel^T p_{N,j}$ and $q_{N,i}\hrel^T
q_{N,j}$. Then there exist $t_1,t_2\in T$ such that
$p_{N,i}=t_1p_{N,j}$ and $q_{N,i}=q_{N,j}t_2$. Then
\begin{equation*}
x_i=p_{N,i}x_Nq_{N,i}=t_1p_{N,j}x_Nq_{N,j}t_2=t_1x_jt_2,
\end{equation*}
and so $J_{x_i}^S=J_{x_j}^S$, a contradiction.
\end{proof}

\begin{thm}\label{th:min_J}
Let $S$ be a semigroup and let $T$ be a subsemigroup of $S$ with
finite Green index. Then $T$ satisfies $\min_J$ if and only if $S$ satisfies
$\min_J$.
\end{thm}

\begin{proof}
Without loss we may assume that $S$ has an identity $1$ and that
$1\in T$.

($\Rightarrow$) Suppose $T$ satisfies $\min_J$, but in $S$ we have
$J_{x_1}^S>J_{x_2}^S>J_{x_3}^S>\cdots$ for some $x_i\in S$. 
As in the
proof of Theorem~\ref{th:min_R} we may assume that $x_i\in T$ for
all $i\geq 1$. By Lemma~\ref{lm:chain} there exists a sequence
$n_1<n_2<n_3<\cdots$ such that $J_{x_{n_1}}^T\geq J_{x_{n_2}}^T\geq
J_{x_{n_3}}^T\geq\cdots$. Therefore $J_{x_k}^T=J_{x_{k+1}}^T$ for some
$k$. Then $J_{x_k}^S=J_{x_{k+1}}^S$, a contradiction.

($\Leftarrow$) Suppose $S$ satisfies $\min_J$, but
$J_{x_1}^T>J_{x_2}^T>J_{x_3}^T>\cdots$ for some $x_i\in T$. As in the
proof of Theorem~\ref{th:min_R} we may assume that
$J_{x_n}^S=J_{x_{n+1}}^S$ for all $n\geq 1$. Then for each $n\geq 1$
there exist $p_n,q_n\in S$ such that $x_n=p_nx_{n+1}q_n$. Define
$p_{i,j}=p_i\cdots p_{j-1}$ and $q_{i,j}=q_{j-1}\cdots q_i$ for all
$1\leq i<j$. Then $x_i=p_{i,j}x_jq_{i,j}$ for all $1\leq i<j$. It
follows that for every $i<j$, either $p_{i,j}\in S\setminus T$, or
$q_{i,j}\in S\setminus T$. By Ramsey's Theorem and up to
renumbering, we may assume that $p_{i,j}\in S\setminus T$ for all
$i<j$. Furthermore, we may even assume that all of $p_{i,j}$ lie in
the same $\hrel^T$-class.

Take arbitrary $i<j$. Then $p_{i,j}p_j=p_{i,j+1}\hrel^T
p_{j,j+1}=p_j$ and so there exists $t\in T$ such that
$p_{i,j}p_j=tp_j$. Then
$x_i=p_{i,j}p_jx_{j+1}q_jq_{i,j}=tp_jx_{j+1}q_jq_{i,j}=tx_jq_{i,j}$
and so $q_{i,j}\in S\setminus T$.

Now, by the Pigeonhole Principle there exist numbers $i<j<N$ such
that $p_{i,N}\hrel^T p_{j,N}$ and $q_{i,N}\hrel^T
q_{j,N}$. Therefore there exist $t_1,t_2\in T$ such that
$p_{i,N}=t_1p_{j,N}$ and $q_{i,N}=q_{j,N}t_2$. Then
\begin{equation*}
x_i=p_{i,N}x_Nq_{i,N}=t_1p_{j,N}x_Nq_{j,N}t_2=t_1x_jt_2,
\end{equation*}
a contradiction. This proves the theorem.
\end{proof}

Another natural finiteness condition, related to (and weaker than) $\min_J$ is that of having a minimal two-sided ideal. The following result is easy to prove, but we include it for completeness:

\begin{prop}\label{pr:obvious}
Let $T$ be a subsemigroup of finite Green index in a semigroup $S$.
If $T$ has a minimal ideal, then $S$ has a minimal ideal.
\end{prop}

\begin{proof}
Let $I$ be a minimal ideal in $T$ and assume that $S$ does not have
a minimal ideal. Take any $x\in I$. Then there exists an infinite
chain $J_{x}^S>J_{x_1}^S>J_{x_2}^S>\cdots$ where $x_i\in S$. As in
the proof of Theorem~\ref{th:min_R} we may assume that $x_i\in T$
for all $T$. Now, $J_{x_1}^T\geq J_{x}^T$ and so $J_{x_1}^S\geq
J_{x}^S$, a contradiction.
\end{proof}

The converse of Proposition \ref{pr:obvious} does not hold. Indeed, if $T$ is any semigroup, the
semigroup $S=T^0$, obtained by adjoining a zero element to $T$,
has $T$ as a subsemigroup of finite Green (and indeed Rees) index, and has $\{0\}$ as its minimal ideal.


\section{All Ideals Have Finite Rees Index}
\label{secfinRees}

In this section we present a result which gives a positive answer to \cite[Open Problem 11.3(ii)]{ruskuc1}. 

\begin{thm}\label{rees}
Let $S$ be a semigroup and let $T$ be a subsemigroup of $S$ with
finite Rees index. If every ideal in $S$ has finite Rees index, then
every ideal in $T$ has finite Rees index.
\end{thm}

\begin{proof}
Suppose that every ideal in $S$ has finite Rees index.
Let $I$ be any ideal in $T$, and set $F = S \setminus T$.
For $i\in I$ define two sets
$$
X_i=\{ f\in F\::\: fi\in I\},\ Y_i=\{ f\in F\::\: if\in I\}.
$$
Let $i_1$ (resp. $i_2$) be any element of $I$ such that the set
$X_{i_1}$ (resp. $Y_{i_2}$) has the maximal possible size.

We claim that
$$
Fi_1I\subseteq F\cup I,\ Ii_2F\subseteq F\cup I.
$$
It suffices to prove the first inclusion; the second is dual.
Suppose that there exist $f\in F$, $j\in I$ such that $fi_1 j\in T\setminus I$.
Then $f\not\in X_{i_1}$ and
for any $j_1\in I$ we have $ fi_1jj_1\in I$.
This implies $X_{i_1jj_1}\supseteq X_{i_1}\cup \{f\}$, contradicting the choice of $i_1$.

Consider now the ideal $J=S^1 i_1i_2 S^1$ of $S$. We have
\begin{eqnarray*}
S^1 i_1i_2 S^1 &=& (F\cup T^1)i_1i_2(F\cup T^1)=
Fi_1i_2F \cup T^1i_1i_2 F \cup Fi_1i_2T^1 \cup T^1i_1i_2T^1\\
&\subseteq& Fi_1i_2F\cup Ii_2F\cup Fi_1I\cup II\subseteq
Fi_1i_2F\cup I\cup F.
\end{eqnarray*}
Note that the set $Fi_1i_2F$ is finite, and so
$I\cap J$ has finite index in $J$.
By assumption $J$ has finite Rees index in $S$. It follows that $I$ has finite Rees index in $S$, and hence in $T$ as well.
\end{proof}

\begin{remark}
\label{remark1}
The converse of Theorem \ref{rees} does not hold: Adjoining a zero to \emph{any} infinite semigroup $S$ results in a semigroup with an ideal (namely $\{0\}$) of infinite Rees index.  
Thus a counterexample to the converse of Theorem \ref{rees} may be obtained by taking an infinite semigroup whose ideals all have finite Rees index (e.g. an infinite group) and adjoining a zero. 
\end{remark}

The analogue of Theorem~\ref{rees} for right ideals also holds:
assume that $T$ is a subsemigroup of finite Rees index in $S$ and
that every right ideal in $S$ has finite Rees index in $S$. Let $R$
be a right ideal in $T$. Take any $r\in R$. Then $rS=rT\cup
r(S\setminus T)\subseteq R\cup r(S\setminus T)$ and the complement
of $rS$ in $S$ must be finite. Hence $T\setminus R$ is finite and so
$R$ has finite Rees index in $T$.

\begin{que}
Can the assumption of $T$ having finite Rees index in Theorem \ref{rees} be weakened to finite Green index?
In other words: if every ideal of $S$ has finite Rees index and if $T$ is a subsemigroup of $S$ of finite Green index, is it necessarily the case that every ideal of $T$ has finite Rees index?
\end{que}

To finish off this section, we prove the following proposition about
the related finiteness condition of every subsemigroup having finite
Green index:

\begin{prop}
If every subsemigroup of a semigroup $S$ has finite Green index in
$S$, then $S$ is finite.
\end{prop}

\begin{proof}
Take any element $a\in S$ and let $T=\langle a\rangle$.
If $T$ is finite, then $S$ is finite too, since there are finitely many $\rrel^T$-classes and each is bounded 
in size by $T$.
Now suppose that $T$ is infinite. Consider the subsemigroup $T^\prime=\langle a^2\rangle$.
Since $T^\prime$ has finite Green index in $S$, we obtain that $a^{2k+1}\hrel^T a^{2n+1}$ for some
$1\leq k<n$. Then $a^{2k+1}=a^{2n+1}\cdot a^{2m}$ for some $m\geq 0$, a contradiction.
\end{proof}


\section{Global Torsion}
\label{secglobal}

For a semigroup $S$ and $n\in\mathbb{N}$ define $S^n=\{ s_1\cdots s_n\::\: s_1,\dots,s_n\in S\}$.
We say that $S$ has \emph{global torsion} if $S^{n+1}=S^n$ for some $n\in\mathbb{N}$.
It is clear that $S\supseteq S^2\supseteq S^3\supseteq\dots$, and it follows that global torsion is a finiteness condition. 

\begin{thm}\label{globalidem}
Let $S$ be a semigroup and let $T$ be a subsemigroup of $S$ with finite Green index.
If $T$ has global torsion, then $S$ has global torsion as well.
\end{thm}

\begin{proof}
Let $m\in \mathbb{N}$ be such that $T^{m+1}=T^m$, and let $r$ be the number of $\hrel^T$-classes in $S\setminus T$. We begin by proving the following:

\begin{lemma} \label{neat}
For any $s_1,\ldots, s_{r+1}\in S$ we either have $s_1\cdots s_{r+1}\in S^{r+2}$ or else $s_1\cdots s_i\in T$ for some $i\in \{ 1,\ldots,r+1\}$.
\end{lemma}

\begin{proof}
Assume that $s_1\cdots s_i\in S\setminus T$ for all $1\leq i\leq r+1$. Then there exist $1\leq i<j\leq r+1$ such that $s_1\cdots s_i\hrel^Ts_1\cdots s_j$. Hence there exists $t\in T^1$ such that
$$s_1\cdots s_i=(s_1\cdots s_i)(s_{i+1}\cdots s_jt)\in S^{i+1}.$$
Hence
\[
s_1\dots s_{r+1}=(s_1\dots s_i)s_{i+1}\dots s_{r+1}
\in S^{i+1} s_{i+1}\dots s_{r+1}\subseteq S^{r+2},
\]
as required.
\end{proof}

Resuming the proof of the theorem, let $n=(r+1)m$. We claim that
$S^{n+1}=S^n$. Clearly $S^{n+1}\subseteq S^n$. Let $s_1,\ldots,
s_n\in S$ be arbitrary, so that $s_1\cdots s_n$ is a typical element
of $S^n$. If for any $i\in\{1,\ldots, n-r\}$ we have
$s_is_{i+1}\cdots s_{i+r}\in S^{r+2}$ then we also have
$s_1s_2\ldots s_n\in S^{n+1}$ and the proof is finished. The
alternative is, by Lemma \ref{neat}, that for every $i\in
\{1,\ldots,n-r\}$ there exists $j\in\{0,\ldots,r\}$ such that
$s_is_{i+1}\cdots s_{i+j}\in T$. In particular, there exist $j_1 < j_2 <\dots<j_m$ belonging to
$\{1,\dots,n\}$ such that $j_{k+1}-j_k\leq r+1$ for all $k=1,\dots,m-1$ and
$$
s_1s_2\cdots s_{j_1}, s_{j_1+1}s_{j_1+2}\cdots s_{j_2},\ldots, s_{j_{m-1}+1}s_{j_{m-1}+2}\ldots s_{j_m} \in T.
$$
But then
$$
s_1s_2\cdots s_{j_m}\in T^m= T^{m+1}=\dots =T^{j_m}=T^{j_{m}+1}\subseteq S^{j_{m}+1}.
$$
Abbreviating $j_m=j$, we now have
\[
s_1\dots s_n=(s_1\dots s_{j})(s_{j+1}\dots s_n)\in S^{j+1} s_{j+1}\dots s_n\subseteq S^{n+1},
\]
completing the proof.
\end{proof}

\begin{remark}
The converse of Theorem \ref{globalidem} does not hold: adjoining an identity element to an arbitrary semigroup $T$ yields a semigroup $S$ such that $S^2=S$.
\end{remark}

\section{Eventual regularity}
\label{secpireg}

We close the paper by discussing one more important finiteness condition, this time not related to ideals.

\begin{definition}
A semigroup $S$ is \emph{eventually regular} if for every $s \in S$ there exists $n
\in \mathbb{N}$ such that $s^n$ is a regular element of $S$.
\end{definition}
The class of eventually regular semigroups (also called $\pi$-regular) was introduced by Edwards in \cite{Edwards1983}. Further results on these semigroups include \cite{Auinger1996, Easdown1984, Higgins1993}. 
Clearly every finite semigroup is eventually regular, i.e. eventual regularity is
a finiteness condition.

\begin{thm}
\label{pireg}
Let $S$ be a semigroup and let $T$ be a subsemigroup with finite Green
index. Then $S$ is eventually regular if and only if $T$ is eventually regular.
\end{thm}
\begin{proof}
Suppose that $T$ is eventually regular and let $s \in S$ be arbitrary. If
$s^m \in T$ for some $m \in \mathbb{N}$, then, since $T$ is eventually regular,
$(s^m)^n = s^{mn}$ is regular in $T$ (and hence also in $S$) for some $n
\in \mathbb{N}$. Otherwise $s^m \not\in T$ for all $m$ and since $T$ has
finite Green index in $S$ there exist $n,r \in \mathbb{N}$ with $s^{n+r}
\hrel^T s^n$.  Then  proof of \cite[Theorem~18]{Gray} choosing $z \in \mathbb{N}$ with $0 \leq z \leq
r-1$ and $n + z \equiv 0 \pmod{r}$ we have
$
(s^{n+z})^2 \hrel^T s^{n+z}.
$
By Proposition \ref{basicproperties} (ii) we have that the relative $\hrel^T$ class of $s^{n+z}$ is a group,
and hence $s^{n+z}$ is a regular element.

For the converse, suppose that $S$ is eventually regular and let $t \in T$.
Since $S$ is eventually regular there exists an infinite subset $I \subseteq
\mathbb{N}$ such that  $t^{i}$ is regular in $S$ for all $i \in I$.
For each $i \in I$ let $s_i$ be an inverse of $t^{i}$ in $S$, so
\begin{equation}
t^{i} s_{i} t^{i} = t^{i} ,\quad s_{i} t^{i} s_i = s_i.
\label{plus}
\end{equation}
If $s_i \in T$ for some $i \in I$ then $t^{i}$ is regular in $T$ and we
are done, so suppose otherwise.
For all $i \in I$, set $f_i = t^{i} s_i$ noting that by \eqref{plus},
$f_i$ is an idempotent satisfying $f_i \rrel^S t^{i}$ and $f_i
\lrel^S s_i$.
Since  $s_i \in S \setminus T$ for all $i \in I$, and $T$ has finite
Green index in $S$, it follows that there is an infinite subset $J
\subseteq I$ such that for all $i,j \in J$ we have $s_i \hrel^T
s_j$. Let $i,j \in J$ be arbitrary, with $i < j$ say. Then
\[
f_i \lrel^S s_i \lrel^S s_j \lrel^S f_j
\]
and therefore $f_i f_j = f_i$. Since $\rrel$ on $S$ is a left
congruence, $ t^j  \rrel^S f_j $  implies $f_i t^j  \rrel^S
f_i f_j$ and hence
\[
t^j = t^i t^{j-i} = t^{i} s_{i} t^{i} t^{j-i} = (t^{i} s_{i}) t^{j} =
f_i t^j  \rrel^S f_i f_j   = f_i  \rrel^S t^i.
\]
By a dual argument $t^j \lrel^S t^i$ and hence $t^j \hrel^S
t^i$.

Since $i,j \in J$ were arbitrary it follows that $t^k \hrel^S t^l$
for all $k,l \in J$.
By \cite[Proposition~10]{Gray} each $\hrel^S$-class of
$S$ is a union of finitely many $\hrel^T$-classes. Since $J$ is
infinite it follows that there exist distinct $p,q \in J$ with $t^p \hrel^T
t^q$. Now as in the proof of the converse above we can find a number $y
\in \mathbb{N}$ with $(t^y)^2 \hrel^T t^y$, and we conclude that
$t^y$ is a regular element of $T$.   \end{proof}


\begin{flushleft}
R. Gray\\
Centro de \'{A}lgebra da Universidade de Lisboa \\ 
Av. Prof. Gama Pinto 2  \\
1649-003 Lisboa,  Portugal\\
\smallskip
\texttt{rdgray@fc.ul.pt}\\
\bigskip
V. Maltcev, J.D. Mitchell, N. Ru\v{s}kuc\\
School of Mathematics and Statistics\\
University of St Andrews\\
St Andrews KY16 9SS\\
Scotland, U.K.\\
\smallskip
\texttt{\{victor,jamesm,nik\}@mcs.st-and.ac.uk}
\end{flushleft}

\end{document}